\documentclass[12pt, amscd]{amsart}
\usepackage{amscd}
\usepackage{verbatim}

\input xy
\xyoption{all}

\usepackage{amssymb}

\DeclareMathAlphabet{\cat}{OT1}{cmss}{m}{sl}

\newtheorem{theorem}{Theorem}[section]

\newtheorem{lemma}[theorem]{Lemma}
\newtheorem{corollary}[theorem]{Corollary}

\theoremstyle{definition}
\newtheorem{remark}[theorem]{Remark}

\newtheorem{example}[theorem]{Example}

\newcommand{\CH}{\operatorname{CH}}

\newcommand{\Ker}{\operatorname{Ker}}

\newcommand{\gSpin}{\operatorname{\mathbf{Spin}}}

\newcommand{\Z}{\mathbb{Z}}

\newcommand{\Q}{\mathbb{Q}}

\usepackage[hypertex]{hyperref}

\title[On the exponent of spinor groups] 
{On the exponent of spinor groups}

\author
[S. Baek] {Sanghoon Baek}

\address
{Department of Mathematics and Statistics, University of Ottawa, Canada}

\email {sbaek@uottawa.ca}

\begin{document}

\maketitle

\section{Introduction}

Let $G$ be a split simple simply connected group of rank $n$ over a field $F$. Fix a maximal split torus $T$ of $G$ and a Borel subgroup $B$ containing $T$. We denote by $W$ the Weyl group of $G$ with respect to $T$. Let $\Lambda$ be the weight lattice of $G$ (hence, $T^{*}=\Lambda$).

We denote by $\omega_{1},\cdots, \omega_n$ the fundamental weights of $\Lambda$. We let $I_{K}:=\Ker(\Z[\Lambda]\to \Z)$ and $I_{CH}:=\Ker(S^{*}(\Lambda)\to \Z)$ be the augmentation ideals, where $\Z[\Lambda]\to \Z$ (respectively, $S^{*}(\Lambda)\to \Z$) is the map from the group ring $\Z[\Lambda]$ (respectively, the symmetric algebra) of $\Lambda$ to the ring of integers by sending $e^{\lambda}$ to $1$ (respectively, any element of positive degree to $0$).

For any $i\geq 0$, we consider the ring homomorphism \[\phi^{(i)}:\Z[\Lambda]\to\Z[\Lambda]/I_{K}^{i+1}\to S^{*}(\Lambda)/I_{CH}^{i+1}\to S^{i}(\Lambda),\]
where the first and the last maps are projections and the middle map sends $e^{\sum_{j=1}^{n}a_j\omega_j}$ to $\prod_{j=1}^{n}(1-\omega_j)^{-a_j}$. The \emph{$i$th-exponent of $G$} (denoted by $\tau_{i}$), as introduced in \cite{BNZ}, is the gcd of all nonnegative integers $N_{i}$ satisfying
\[N_i\cdot (I_{CH}^W)^{(i)} \subseteq \phi^{(i)}(I_{K}^W),
\]
where $I_{K}^W:=\langle \Z[\Lambda]^W\cap I_{K}\rangle$ (respectively, $I_{CH}^W:=\langle S^{*}(\Lambda)^W\cap I_{CH}\rangle$) denotes the $W$-invariant augmentation ideal of $\Z[\Lambda]$ (respectively, $S^{*}(\Lambda)$) and $(I_{CH}^W)^{(i)}=I_{CH}^W\cap S^i(\Lambda)$. Informally, these numbers $\tau_i$ measure how far is the ring $S^{*}(\Lambda)^{W}$ from being a polynomial ring in basic invariants.

For any $i\leq 4$, it was shown that the $i$th-exponent of $G$ divides the Dynkin index in \cite{BNZ} and this integer was used to estimate the torsion of the Grothendieck gamma filtration and the Chow groups of $E/B$, where $E/B$ denotes the twisted form of the variety of Borel subgroups $G/B$ for a $G$-torsor $E$.

In this paper, we show that all the remaining exponents of spinor groups divide the Dynkin index $2$. 

\smallskip

\paragraph{\bf Acknowledgments.} The work has been partially supported from the Fields Institute and from Zainoulline's NSERC Discovery grant 385795-2010.

\section{Exponent}

Let $G$ be $\gSpin_{2n+1}$ ($n\geq 3$) or $\gSpin_{2n}$ ($n\geq 4$). The
fundamental weights are defined by
\begin{align*} \label{fundamental weights}
\omega_{1}&=e_{1}, \omega_{2}=e_{1}+e_{2}, \cdots, \omega_{n-1}=e_{1}+\cdots+e_{n-1}, \omega_{n}=\frac{e_{1}+\cdots+e_{n}}{2},\\
\omega_{1}&=e_{1}, \omega_{2}=e_{1}+e_{2}, \cdots, \omega_{n-1}=\frac{e_{1}+\cdots+e_{n-1}-e_{n}}{2}, \omega_{n}=\frac{e_{1}+\cdots+e_{n}}{2},
\end{align*}
respectively, where the canonical
basis of $\mathbb{R}^{n}$ is denoted by $e_{i}$ ($1\leq i \leq n$).

For $1\leq i \leq n$, let
\begin{equation}\label{basicinv}
q_{2i}:=e_{1}^{2i}+\cdots +e_{n}^{2i}
\end{equation}   
be the basic invariants of the group $G$, i.e., be algebraically independent homogeneous generators of $S^*(\Lambda)^{W}$ as a $\Q$-algebra (see \cite[\S3.5 and \S3.12]{Hum}), together with
\begin{equation}\label{Dndegn}
q'_{n}:=e_{1}\cdots e_{n}
\end{equation}
if $G=\gSpin_{2n}$.

For any $\lambda\in \Lambda$, we denote by $W(\lambda)$ the $W$-orbit of $\lambda$. For any finite set $A$ of weights, we denote $-A$ the set of opposite weights.

The Weyl groups of $\gSpin_{2n+1}$ and $\gSpin_{2n}$  are $(\Z/2\Z)^{n}\rtimes S_{n}$
and $(\Z/2\Z)^{n-1}\rtimes S_n$, respectively. Hence, by the
action of these Weyl groups, one has the following
decomposition of $W$-orbits: if $G=\gSpin_{2n+1}$ (respectively, $G=\gSpin_{2n}$), then for any $1\leq k \leq n-1$ (respectively, $1\leq k\leq n-2$) 
\begin{equation}\label{orbitwk}
W(\omega_{k})=W_{+}(\omega_{k})\cup -W_{+}(\omega_{k}),
\end{equation}
where $W_{+}(\omega_{k})=\{e_{i_1}\pm \cdots \pm e_{i_{k}}\}_{i_{1}<\cdots <i_{k}}$. If $n$ is even, then the $W$-orbits of the last two fundamental weights of $\gSpin_{2n}$ are given by
\begin{equation}\label{orbitwn}
W(\omega_{n-1})=W_{+}(\omega_{n-1})\cup -W_{+}(\omega_{n-1}) \text{ and } W(\omega_{n})=W_{+}(\omega_{n})\cup -W_{+}(\omega_{n}),
\end{equation}
where $W_{+}(\omega_{n-1})$ (respectively, $W_{+}(\omega_{n})$) is the subset of $W(\omega_{n-1})$ (respectively, $W(\omega_{n})$) containing elements of the positive sign of $e_1$.

For any $\lambda=\sum_{j=1}^{n}a_j\omega_{j}\in \Lambda$ and any integer $m\geq 0$, we set $\lambda(m)=\sum_{j=1}^{n}a_j\omega_{j}^{m}$. For example, $\lambda(0)=\sum_{j=1}^{n}a_j$ and $\lambda(1)=\lambda$. We shall need the following lemma:

\begin{lemma}\label{elemenlemma}
$(i)$ If $G$ is $\gSpin_{2n+1}$ $($respectively, $\gSpin_{2n}$$)$, then for any odd integer $p$, any nonnegative integers $m_1,\cdots, m_p$ and, any $1\leq k\leq n-1$ $($respectively, any $1\leq k\leq n-2$$)$, we have
\[\sum_{\lambda\in W(\omega_{k})}\lambda(m_1)\cdots \lambda(m_{p})=0.
\]

\smallbreak

$(ii)$ If $G$ is $\gSpin_{2n}$ with odd $n$, then for any even integer $p$ and any nonnegative integers $m_1,\cdots, m_p$, we have
\[\sum_{\lambda\in W(\omega_{n})}\lambda(m_1)\cdots \lambda(m_{p})=\sum_{\lambda\in W(\omega_{n-1})}\lambda(m_1)\cdots \lambda(m_{p}).\]

$(iii)$ If $G$ is $\gSpin_{2n}$, then for any odd integer $p<n$ and any nonnegative integers $m_1,\cdots, m_p$, we have
\[\sum_{\lambda\in W(\omega_{n})}\lambda(m_1)\cdots \lambda(m_{p})=\sum_{\lambda\in W(\omega_{n-1})}\lambda(m_1)\cdots \lambda(m_{p})=0.\]
\end{lemma}
\begin{proof}
$(i)$ It follows from (\ref{orbitwk}) that
\begin{align*}
\sum_{\lambda\in W(\omega_{k})}\lambda(m_1)\cdots \lambda(m_{p})&=\sum_{\lambda\in W_{+}(\omega_{k})}\lambda(m_1)\cdots \lambda(m_{p})+\sum_{\lambda\in -W_{+}(\omega_{k})}\lambda(m_1)\cdots \lambda(m_{p})\\
&=\sum_{\lambda\in W_{+}(\omega_{k})}\lambda(m_1)\cdots \lambda(m_{p})-\sum_{\lambda\in W_{+}(\omega_{k})}\lambda(m_1)\cdots \lambda(m_{p})\\
&=0.
\end{align*}

\smallskip

$(ii)$ If $G$ is $\gSpin_{2n}$ with odd $n$, then we have $W(\omega_{n})=-W(\omega_{n-1})$. Hence, the result immediately follows from the assumption that $p$ is even.

\smallskip

$(iii)$ If $n$ is even, then the result follows from (\ref{orbitwn}) by the same argument as in the proof of $(i)$. In general, note that for any $\lambda_{i_{1}},\cdots, \lambda_{i_{p}}\in W_{+}(\omega_{1})$ the term $\lambda_{i_{1}}(m_1)\cdots \lambda_{i_{p}}(m_{p})/2^{p}$ (respectively, -$\lambda_{i_{1}}(m_1)\cdots \lambda_{i_{p}}(m_{p})/2^{p}$) appears $2^{n-2}$ times (respectively, $2^{n-2}$) in  both sums in $(iii)$.
\end{proof}

Let $p$ be an even integer and $q\geq 2$ an integer. For any nonnegative integers $m_1,\cdots, m_p$, we define
\[\Lambda(p,q)(m_1,\cdots,m_p):=\sum \lambda_{j_1}(m_1)\cdots \lambda_{j_p}(m_{p}),\]
where the sum ranges over all different $\lambda_{i_1},\cdots,\lambda_{i_{q}}\in W_{+}(\omega_1)$ and all $\lambda_{i_1},\cdots,\lambda_{i_{p}}\\\in \{\lambda_{i_1},\cdots,\lambda_{i_{q}}\}$ such that the numbers of $\lambda_{i_1},\cdots,\lambda_{i_{q}}$ appearing in $\lambda_{i_1},\cdots,\lambda_{i_{p}}$ are all nonnegative even solutions of $x_1+\cdots+x_q=p$. If $p<2q$, then we set $\Lambda(p,q)(m_1,\cdots,m_p)=0$. Given $m_1,\cdots,m_p$, we simply write $\Lambda(p,q)$ for $\Lambda(p,q)(m_1,\cdots,m_p)$.

For instance, $\Lambda(4,2)$ is the sum of $\lambda_{j_{1}}(m_1)\lambda_{j_{2}}(m_2)\lambda_{j_{3}}(m_3)\lambda_{j_{4}}(m_4)$ for all $j_{1}, j_{2}, j_{3}, j_{4}\in \{i,j\}$ and all $1\leq i\neq j\leq n$ such that two $i$'s and two $j$'s appear in $j_{1}, j_{2}, j_{3}, j_{4}$.

\begin{example}\label{omega2eg}
We observe that
\begin{equation}\label{simpleob}
(x_1+x_2)(x'_1+x'_2)+(x_1-x_2)(x'_1-x'_2)=2(x_1x'_1+x_2x'_2).
\end{equation}
If $G$ is $\gSpin_{2n+1}$ or $\gSpin_{2n}$, then by (\ref{orbitwk}) and (\ref{simpleob}) we have
\[\sum_{W_{+}(\omega_2)}\lambda(m_1)\lambda(m_2)=2(n-1)\sum_{W_{+}(\omega_1)}\lambda(m_1)\lambda(m_2)\]
for any nonnegative integers $m_1$ and $m_2$ as we have $(n-1)$ choices of such pairs in the left hand side of (\ref{simpleob}) from $W_{+}(\omega_2)$, which implies that
\[\sum_{W(\omega_2)}\lambda(m_1)\lambda(m_2)=2(n-1)\sum_{W(\omega_1)}\lambda(m_1)\cdots\lambda(m_2),\]
(cf. \cite[Lemma 5.1(ii)]{BNZ}). For any even $p\geq 4$, we apply the same argument with the expansion of $(x_1+x_2)\cdots(x_{1}^{(p)}+x_{2}^{(p)})+(x_1-x_2)\cdots(x_{1}^{(p)}-x_{2}^{(p)})$. Then, we have
\[\sum_{W_{+}(\omega_2)}\lambda(m_1)\cdots\lambda(m_p)=2(n-1)\sum_{W_{+}(\omega_1)}\lambda(m_1)\cdots\lambda(m_p)+2\Lambda(p,2),\]
which implies that
\[\sum_{W(\omega_2)}\lambda(m_1)\cdots\lambda(m_p)=2(n-1)\sum_{W(\omega_1)}\lambda(m_1)\cdots\lambda(m_p)+2^{2}\Lambda(p,2).\]
\end{example}

We generalize Example \ref{omega2eg} to any $\omega_{k}$ as follows.

\begin{lemma}\label{generallem}
If $G$ is $\gSpin_{2n+1}$ $($respectively, $\gSpin_{2n}$$)$, then for any $1\leq k \leq n-1$ $($respectively, $1\leq k\leq n-2$$)$, any even $p$, and any nonnegative integers $m_1,\cdots m_p$ we have
\[\sum_{W(\omega_{k})}\lambda(m_1)\cdots \lambda(m_p)=2^{k-1}{{n-1}\choose{k-1}}\sum_{W(\omega_{1})}\lambda(m_1)\cdots \lambda(m_p)+\sum_{j=2}^{k}2^{k}{{n-j}\choose{k-j}}\Lambda(p,j).
\]
\end{lemma}
\begin{proof}
For any $\lambda \in W(\omega_1)$, there are $2^{k}{{n-1}\choose{k-1}}$ choices of the element containing $\lambda$ in $W(\omega_k)$, thus we have the term $2^{k-1}{{n-1}\choose{k-1}}\sum_{W(\omega_{1})}\lambda(m_1)\cdots \lambda(m_p)$ in $\sum_{W(\omega_{k})}\lambda(m_1)\cdots \lambda(m_p)$.

If an element $\lambda\in W(\omega_1)$ appears odd times in a term $\lambda_{i_1}(m_1)\cdots \lambda_{i_p}(m_p)$ of $\sum_{W(\omega_{k})}\lambda(m_1)\cdots \lambda(m_p)$, where $\lambda_{i_1},\cdots,\lambda_{i_p}\in W(\omega_1)$, then by the action of Weyl group this term vanishes in $\sum_{W(\omega_{k})}\lambda(m_1)\cdots \lambda(m_p)$. Hence, the remaining terms in $\sum_{W(\omega_{k})}\lambda(m_1)\cdots \lambda(m_p)$ are a linear combination of $\Lambda(p,j)$ for all $2\leq j\leq k$ such that $p\geq 2k$. As each term $\Lambda(p,j)$ appears $2^{k}{{n-j}\choose{k-j}}$ times in $\sum_{W(\omega_{k})}\lambda(m_1)\cdots \lambda(m_p)$, the result follows.
\end{proof}

\smallskip

For any $\lambda\in \Lambda$, we denote by $\rho(\lambda)$ the sum of all elements $e^{\mu}\in \Z[\Lambda]$ over all elements $\mu$ of $W(\lambda)$. Let $i!\cdot \phi^{(i)}(e^\lambda)=\lambda^{i}+S_{i}$ for any $i\geq 1$, where $S_i$ is the sum of remaining terms in $i!\cdot \phi^{(i)}(e^\lambda)$ and $\lambda=\sum a_{j}\omega_j$, $a_{j}\in \Z$. Hence, for any fundamental weight $\omega_k$ we have
\begin{equation}\label{bsifactorial}
i!\cdot \phi^{(i)}(\rho(\omega_{k}))=\sum_{W(\omega_{k})}\lambda^{i}+\sum_{W(\omega_{k})}S_{i}.
\end{equation}
We view $i!\cdot \phi^{(i)}(e^\lambda)$ as a polynomial in variables $\lambda, \lambda(m_1),\cdots ,\lambda(m_j)$ for some nonnegative integers $m_1,\cdots, m_j$. Let $T_{i}$ be the sum of monomials in $S_i$ whose degrees are even. 

If $G$ is $\gSpin_{2n+1}$ $($respectively, $\gSpin_{2n}$$)$, then by Lemma \ref{elemenlemma}(i) the equation (\ref{bsifactorial}) reduces to
\begin{equation}\label{bsifactorial2}
i!\cdot \phi^{(i)}(\rho(\omega_{k}))=\sum_{W(\omega_{k})}\lambda^{i}+\sum_{W(\omega_{k})}T_{i}.
\end{equation}
for any $1\leq k\leq n-1$ $($respectively $1\leq k\leq n-2$$)$.

Given $p$ and $q$, we define
\[\Omega(p,q):=\sum\Lambda(p,q)(m_1,\cdots,m_p),\]
where the sum ranges over all $m_1,\cdots,m_p$ which appear in all monomials of $T_i$.
 
\begin{example}\label{deg6phi}

$(i)$ If $G$ is $\gSpin_{2n+1}$ or $\gSpin_{2n}$ and $i=4$, then by (\ref{bsifactorial2}) and Lemma \ref{generallem} we have
\begin{align*}
4!\phi^{(4)}(\rho(\omega_{1}))&=\sum_{W(\omega_{1})}\lambda^{4}+\sum_{W(\omega_{1})}T_{4},\\
4!\phi^{(4)}(\rho(\omega_{2}))&=\sum_{W(\omega_{2})}\lambda^{4}+\sum_{W(\omega_{2})}T_{4}\\
&=\sum_{W(\omega_{2})}\lambda^{4}+2(n-1)\sum_{W(\omega_{1})}T_{4},
\end{align*}
which implies that
\[4!(\phi^{(4)}(\rho(\omega_{2}))-2(n-1)\phi^{(4)}(\rho(\omega_{1})))=\sum_{W(\omega_{2})}\lambda^{4}-2(n-1)\sum_{W(\omega_{1})}\lambda^{4}.\]
By Lemma \ref{generallem}, the right-hand side of the above equation is equal to
\[4\Lambda(4,2)=4\cdot \frac{4!}{2!2!}\sum_{i<j}e_{i}^{2}e_{j}^{2}.\]
Hence, we have
\[\phi^{(4)}(\rho(\omega_{2}))-2(n-1)\phi^{(4)}(\rho(\omega_{1}))=\sum_{i<j}e_{i}^{2}e_{j}^{2}.\]

\smallskip

$(ii)$ If $G$ is $\gSpin_{2n+1}$ ($n\geq 4$) or $\gSpin_{2n}$ ($n\geq 5$) and $i=6$, then by (\ref{bsifactorial2}) and Lemma \ref{generallem} we have
\begin{align*}
6!\phi^{(6)}(\rho(\omega_{1}))&=\sum_{W(\omega_{1})}\lambda^{6}+\sum_{W(\omega_{1})}T_{6},\\
6!\phi^{(6)}(\rho(\omega_{2}))&=\sum_{W(\omega_{2})}\lambda^{6}+2(n-1)\sum_{W(\omega_{1})}T_{6}+4\Omega(4,2),\\
6!\phi^{(6)}(\rho(\omega_{3}))&=\sum_{W(\omega_{3})}\lambda^{6}+4{{n-1}\choose{2}}\sum_{W(\omega_{1})}T_{6}+8(n-2)\Omega(4,2),
\end{align*}
which implies that
\[\phi^{(6)}(\rho(\omega_{3}))-2(n-2)\phi^{(6)}(\rho(\omega_{2}))+2(n-1)(n-2)\phi^{(6)}(\rho(\omega_{1}))=\sum_{i<j<k}e_{i}^{2}e_{j}^{2}e_{k}^{2}.\]

\end{example}

\begin{lemma}\label{Dnqprime}
$(i)$ If $G$ is $\gSpin_{2n}$, then we have 
\[
\sum_{W(\omega_n)}\lambda^{n} - \sum_{W(\omega_{n-1})}\lambda^{n}=n!e_1\cdots e_n.
\]

\smallskip

$(ii)$ If $G$ is $\gSpin_{2n}$, then for any $1\leq p\leq n-1$ and any nonnegative integers $m_1,\cdots,m_p$ we have 
\[
\sum_{W(\omega_n)}\lambda(m_1)\cdots \lambda(m_p)=\sum_{W(\omega_{n-1})}\lambda(m_1)\cdots \lambda(m_p).
\]
\end{lemma}
\begin{proof}
$(i)$ First, assume that $n\geq 4$ is even. We show that \[
\sum_{W_{+}(\omega_n)}\lambda^{n} - \sum_{W_{+}(\omega_{n-1})}\lambda^{n}=(n!/2)e_1\cdots e_n.
\]
As $|W_{+}(\omega_n)|=|W_{+}(\omega_{n-1})|=2^{n-2}$, we have
\[
(n!/2^{n})2^{n-2}e_1\cdots e_n - (-(n!/2^{n})2^{n-2}e_1\cdots e_n)=(n!/2)e_1\cdots e_n
\]
in $\sum_{W(\omega_n)}\lambda^{n} - \sum_{W(\omega_{n-1})}\lambda^{n}$. If one of the exponents $i_{1},\cdots,i_{n}$ in $e_1^{i_1}\cdots e_n^{i_n}$ (except the case $i_1=\cdots=i_n=1$) is odd, then from the orbits $W_{+}(\omega_n)$ and $W_{+}(\omega_{n-1})$ this monomial vanishes in each sum of $\sum_{W_{+}(\omega_n)}\lambda^{n} - \sum_{W_{+}(\omega_{n-1})}\lambda^{n}$. Otherwise, the terms $2^{n-2}\sum_{j=1}^{n}e_{j}^{n} ,\Lambda(n,2)\cdots, \Lambda(n,n/2)$ with $m_1=\cdots=m_n=1$ are in both $\sum_{W_{+}(\omega_n)}\lambda^{n}$ and $\sum_{W_{+}(\omega_{n-1})}\lambda^{n}$.

Now, we assume that $n\geq 4$ is odd. As $|W(\omega_n)|=|W(\omega_{n-1})|=2^{n-1}$, we have
\[
(n!/2^{n})2^{n-1}e_1\cdots e_n - (-(n!/2^{n})2^{n-1}e_1\cdots e_n)=n!e_1\cdots e_n
\]
in $\sum_{W(\omega_n)}\lambda^{n} - \sum_{W(\omega_{n-1})}\lambda^{n}$. By the same argument, if one of the exponents $i_{1},\cdots,i_{n}$ in $e_1^{i_1}\cdots e_n^{i_n}$ (except the case $i_1=\cdots=i_n=1$) is odd, then this monomial vanishes in each sum of $\sum_{W(\omega_n)}\lambda^{n} - \sum_{W(\omega_{n-1})}\lambda^{n}$. This completes the proof of $(i)$.

\smallskip

$(ii)$ By Lemma \ref{elemenlemma}(ii)(iii), it is enough to consider the case where both $n$ and $p$ are even. For any $p$ and any $n\geq p+2$, we have
$2^{n-2}(\sum_{W_{+}(\omega_1)}\lambda(m_1)\cdots \lambda(m_{p}))$ in both $\sum_{W_{+}(\omega_n)}\lambda(m_1)\cdots \lambda(m_{p})$ and $\sum_{W_{+}(\omega_{n-1})}\lambda(m_1)\cdots \lambda(m_{p})$. By the action of Weyl group, any term $\lambda_{i_1}(m_1)\cdots \lambda_{i_p}(m_p)$, where an element $\lambda\in W(\omega_1)$ appears odd times in either $\sum_{W_{+}(\omega_n)}\lambda(m_1)\cdots \lambda(m_{p})-2^{n-2}(\sum_{W_{+}(\omega_1)}\lambda(m_1)\\\cdots \lambda(m_{p}))$ or $\sum_{W_{+}(\omega_{n-1})}\lambda(m_1)\cdots \lambda(m_{p})-2^{n-2}(\sum_{W_{+}(\omega_1)}\lambda(m_1)\cdots \lambda(m_{p}))$, vanishes. As each term of $\Lambda(p,2),\cdots, \Lambda(p,p/2)$ appears in both $\sum_{W_{+}(\omega_n)}\lambda(m_1)\cdots\\ \lambda(m_{p})$ and $\sum_{W_{+}(\omega_{n-1})}\lambda(m_1)\cdots \lambda(m_{p})$, this completes the proof.

\end{proof}

\begin{theorem}\label{mainthm}
If $G$ is $\gSpin_{2n+1}$ $($respectively, $\gSpin_{2n}$$)$, then for any $i\geq 3$ and any $n\geq [i/2]+1$ $($respectively, $n\geq [i/2]+2$$)$ the exponent $\tau_i$ divides the Dynkin index $\tau_2=2$.
\end{theorem}
\begin{proof}
As $B_2=C_2$ and $D_3=A_3$, we have $1=\tau_3\mid 2$ by \cite[Theorem 5.4]{BNZ}. If $G$ is $\gSpin_{2n}$ for any $n\geq 4$, then by Lemma \ref{Dnqprime}(i)(ii) we have
\[q'_n=\phi^{(n)}(\rho(\omega_{n}))-\phi^{(n)}(\rho(\omega_{n-1})),
\]
which implies that the invariant $q'_n$ is in the ideal generated by the image of $\phi^{(n)}$. As there are no invariants of odd degree except $q'_n$, we have
\[\tau_{2i+1}\mid \tau_{2i}\]
for all $i\geq 1$. Therefore, it suffices to show that $\tau_{2i}\mid \tau_{2}$ for any $i\geq 2$. 

By Lemma \ref{generallem} together with the same argument as in Example \ref{deg6phi} we have
\begin{equation}\label{complicatedfor}
\phi^{(2i)}(\rho(\omega_{i}))+\sum_{j=1}^{i-1}a_{j}\phi^{(2i)}(\rho(\omega_{i-j}))=\sum_{j_{1}<\cdots<j_{i}}e_{j_1}^{2}\cdots e_{j_i}^{2},
\end{equation}
where the integers $a_1,\cdots, a_{i-1}$ satisfy
\[\Big(\sum_{j=k}^{i-2}2^{j+1}{{n-1-k}\choose{j-k}}a_{j+1}\Big)+2^{i}{{n-1-k}\choose{i-1-k}}=0,\]
for $0\leq k\leq i-2$. Let $p_i$ be the right-hand side of (\ref{complicatedfor}). Then this equation implies that $p_i$ is in the image of $\phi^{(2i)}$. 

We show that the invariant $q_{2i}$ is in the ideal $\phi^{(2i)}(I_K^W)$ for any $i\geq 2$. We proceed by induction on $i$.
As $q_{2}=\phi^{(2)}(\rho(\omega_1))$, the case $i=2$ is obvious. By Newton's identities we have
\begin{equation}\label{newton}
(-1)^{i-1}q_{2i}=ip_{i}-\sum_{j=1}^{i-1}(-1)^{j-1}p_{i-1-j}q_{2j}
\end{equation}
with $p_{0}=1$. By the induction hypothesis, the sum of (\ref{newton}) is in $\phi^{(2i)}(I_K^W)$. Hence, $q_{2i}$ is in $\phi^{(2i)}(I_K^W)$. 
\end{proof}



For any nonnegative integer $n$, we denote by $v_{2}(n)$ the $2$-adic valuation of $n$. For a smooth projective variety $X$ over $F$, we denote by $\Gamma^{*}K(X)$ the gamma filtration on the Grothendieck ring $K(X)$. We let $c_{CH}:S^{*}(\Lambda)\to CH(G/B)$ be the characteristic map.

\begin{corollary}\label{gammatorsion}
Let $G$ be $\gSpin_{2n}$ $($respectively, $\gSpin_{2n+1}$$)$. If $2^{m(i)}(\ker c_{CH})^{(i)}\subseteq(I_{CH}^W)^{(i)}$ for some nonnegative integer $m(i)$, then for any $i\geq 3$ and any $n\geq [i/2]+2$ $($respectively, $n\geq [i/2]+1$$)$ the torsion of $\Gamma^{i}K(G/B)/\Gamma^{i+1}K(G/B)$ is annihilated by $2^{g(i)}$, where $g(i)=1+m(i)+v_{2}((i-1)!)$.
\end{corollary}

\begin{remark}
It is shown that $m(3)=0$ and $m(4)=1$ in \cite[Lemma 6.4]{BNZ}.
\end{remark}

\begin{proof}
The proof of \cite[Theorem 6.5]{BNZ} still works with Theorem \ref{mainthm}.
\end{proof}

\begin{corollary}
Let $G$ be $\gSpin_{2n}$ $($respectively, $\gSpin_{2n+1}$$)$. If $2^{m(i)}(\ker c_{CH})^{(i)}\subseteq(I_{CH}^W)^{(i)}$ for some nonnegative integer $m(i)$, then for any $G$-torsor $E$, any $i\geq 3$ and any $n\geq [i/2]+2$ $($respectively, $n\geq [i/2]+1$$)$ the torsion of $\CH^{i}(E/B)$ is annihilated by $2^{t(i)}$, where $t(i)=1+\sum_{j=3}^{i}g(j)+v_{2}((i-1)!)$.
\end{corollary}
\begin{proof}
By \cite[Theorem 2.2(2)]{Panin}, we have \[\Gamma^{i}K(G/B)/\Gamma^{i+1}K(G/B)\simeq \Gamma^{i}K(E/B)/\Gamma^{i+1}K(E/B).\] As the torsion of $\CH^{i}(E/B)$ is annihilated by \[(i-1)!\prod_{j=1}^{i}e(\Gamma^{i}K(E/B)/\Gamma^{i+1}K(E/B)),\] where $e(\Gamma^{i}K(E/B)/\Gamma^{i+1}K(E/B))$ denotes the finite exponent of its torsion subgroup (see \cite[p.149]{BNZ}), the result follows from Corollary \ref{gammatorsion}.
\end{proof}

\end{document}